\documentclass[letterpaper, 10 pt, conference]{ieeeconf}  

\IEEEoverridecommandlockouts                              
\usepackage{caption}
\title{\LARGE \bf
  Continuity and approximability of competitive spectral radii
}

\author{Marianne Akian, St\'ephane Gaubert, Lo\"\i c Marchesini, Ian Morris
  \thanks{MA, SG, LM are with INRIA, CMAP, Ecole Polytechnique, IP Paris, CNRS, Palaiseau, France; IM is with Queen Mary University of London}
\\[1mm]
{\tt\small \{marianne.akian,stephane.gaubert,loic.marchesini\}@inria.fr, i.morris@qmul.ac.uk}
}
\let\labelindent\relax
\usepackage{enumitem}

\usepackage{dsfont,amsmath,amssymb,mathtools,amsthm}
\usepackage{mathrsfs}

\DeclareMathOperator*{\Int}{Int}

\newcommand{\Lip}{\operatorname{Lip}}
\newcommand{\CLip}{\operatorname{CLip}_1}
\newcommand{\End}{\operatorname{End}}
\newcommand{\hausdorff}{\delta_{\mathscr{H}}}

\newcommand{\shapc}{F}
\newcommand{\N}{\mathbb{N}}
\newcommand{\specrad}{r}
\newcommand{\R}{\mathbb{R}}

\newcommand{\strategyMax}{\mathcal{B}^s}
\newcommand{\strategyMin}{\mathcal{A}^s}
\newcommand{\Min}{Min }
\newcommand{\Max}{Max }
\newcommand{\payoff}{J}
\newcommand{\Funk}{\operatorname{Funk}}

\newcommand{\Hil}{\operatorname{Hil}}
\newcommand{\A}{\mathcal{A}}
\newcommand{\B}{\mathcal{B}}

\usepackage{hyperref}
\usepackage[capitalise]{cleveref}

\newtheorem{prop}{Proposition}[section]
\newtheorem{proposition}[prop]{Proposition}
\newtheorem{theorem}[prop]{Theorem}
\newtheorem{Definition}[prop]{Definition}
\newtheorem{corollary}[prop]{Corollary}

\newtheorem{Theorem}[prop]{Theorem}

\newtheorem{Proposition}[prop]{Proposition}

\newtheorem{hypo}{Assumption}[section]

\theoremstyle{remark}
\newtheorem{remark}{Remark}[section]
\newtheorem{example}{Example}[section]

\DeclareMathSymbol{\mlq}{\mathord}{operators}{``}
\DeclareMathSymbol{\mrq}{\mathord}{operators}{`'}

\makeatletter
\newcommand*\frob{\mathpalette\bigcdot@{.7}}
\newcommand*\bigcdot@[2]{\mathbin{\vcenter{\hbox{\scalebox{#2}{$\m@th#1\bullet$}}}}}

\usepackage[disable,colorinlistoftodos,bordercolor=orange,backgroundcolor=orange!20,linecolor=orange,textsize=scriptsize]{todonotes}
\usepackage{bookmark}
\usepackage[font=small,justification=centering]{caption}
\usepackage[justification=centering]{subcaption}

\usepackage{textcomp}
\usepackage{algorithm}
\usepackage{algorithmic}

\def\namedlabel#1#2{\begingroup
    #2%
    \def\@currentlabel{#2}%
    \phantomsection\label{#1}\endgroup
}

\newcounter{savealgorithm}

\usepackage[sorting=none,natbib=true,backend=biber,doi=false,url=false,
isbn=false,eprint=true,giveninits=true,maxbibnames=99]{biblatex}
\begin{document}

\maketitle
\thispagestyle{empty}
\pagestyle{empty}

\begin{abstract}
  The competitive spectral radius extends
  the notion of joint spectral radius to the two-player case:
  two players alternatively select matrices in prescribed compact sets, resulting in an infinite matrix product; one player wishes to maximize the growth rate of this product, whereas the other player wishes to minimize it.
  We show that when the matrices represent linear operators preserving a cone and satisfying a
  ``strict positivity'' assumption, the competitive spectral radius depends continuously --- and even
  in a Lipschitz-continuous way --- on the matrix sets. Moreover, we show that the competive spectral radius can be approximated up to any accuracy. This relies on the solution of a discretized infinite dimensional non-linear eigenproblem. We illustrate the approach with an example
  of age-structured population dynamics. \end{abstract}
\newcommand{\E}{\mathbb{E}}
\newcommand{\proba}{\mathbb{P}}
\newcommand{\cX}{\mathcal{X}}
\newcommand{\relint}{\operatorname{relint}}
\newcommand{\interior}{\operatorname{int}}
\newcommand{\eigenvect}{v}
\newcommand{\eigenvectw}{w}

\section{Introduction}

\subsection{Motivation}
Asarin,  Cervelle,
Degorre, Dima, Horn and Kozyakin introduced \emph{matrix multiplication games} in~\cite{asarin_entropy}, a game-theoretic framework to extend the joint spectral radius to two sets of matrices $\A$ and $\B$. In this repeated zero-sum game, two players, \Min and \Max, with opposing interests alternatively choose matrices $A_k\in \A$ and $B_k \in \B$ in order to either minimize or maximize the following growth rate:
\[
\limsup_{k \to \infty}\lVert A_1 B_1 \cdots A_k B_k \rVert^{1/k} \enspace .
\]
Their original motivation was to bound the topological entropy of transition systems arising in
automata theory.

Asarin et al.\ proved the existence of the value of the game in the special case of sets of nonnegative matrices satisfying a ``rectangularity'' or ``row interchange'' condition, encompassed by the class of {\em entropy games}, further studied by Akian, Gaubert, Grand-Cl\'ement and Guillaud in~\cite{akian2019}. Akian, Gaubert and Marchesini studied the general case (without rectangularity) in \cite{AGM2024}, in the broader setting of \emph{escape rate games}. Here, \Min and \Max control a switched dynamical system in which the dynamics are nonexpansive with respect to a hemi-metric. They respectively want to minimize and maximize the escape rate of the system. Escape rate games have a value, called \emph{competitive spectral radius}~\cite{AGM2024} -- the joint spectral radius being covered by the one-player case (minimizer-free). The {\em lower spectral radius}~\cite{Bochi2014}, also known as the {\em joint spectral subradius}~\cite{Jungers_2009},
corresponds to the one-player maximizer-free case.

Several questions remain open. One of them being about the continuity of the competitive spectral radius as a function of the model parameters (for instance as a function of the sets of matrices). Whereas the joint spectral radius is continuous in the parameters~\cite{Heil1995}, the lower spectral radius can be only upper-semicontinuous~\cite{Jungers_2009}. Hence, the continuity of the competitive spectral spectral radius must require
restrictive conditions. Another question is to approximate the competitive spectral radius.

\subsection{Contributions}

We provide a general condition which guarantees the continuity of the competitive spectral radius, as a function of the model parameters (\Cref{th-continuity}).
In the special case of matrix multiplication games, this condition requires the matrices
to preserve the interior of a closed convex pointed cone and to satisfy a ``strict positivity''
assumption (\Cref{cor-continuity}).
This condition not only ensures the continuity of the competitive spectral radius as a function of the matrix sets, with respect to the Hausdorff distance, but also makes it lipschitizian. In particular, the competitive spectral radius of matrix multiplication games with positive matrices is continuous.  

Next, under the same assumptions, we develop an algorithm allowing one to approximate
the competitive spectral radius up to an arbitrary precision. This is based on
the characterization of the competitive spectral radius as the unique eigenvalue
of a non-linear ``cohomological'' equation, involving $1$-Lipschitz functions
on a cross-section of a cone. An appropriate discretization of this equation
leads to lower and upper bounds for the competitive spectral radius,
converging to the exact value as the meshsize tends to zero
(\Cref{th-effective}). Moreover, we show that the discretized problem can be
solved by reduction to a repeated game with a finite state space, which
we solve by an algorithm coupling the idea of relative value iteration and Krasnoselskii-Mann damping.
We show that the algorithm terminates after $O(1/h^2)$ iterations, where $h$ is the
required precision (\Cref{th-algo}). Every iteration requires
$O(1/h^{2(d-1)})$ arithmetics
operations, hence a total complexity in $O(1/h^{2d})$,
where $d$ is the dimension
of the cone.

We illustrate our algorithm by applying it to a model of population dynamics.

The organisation of the paper is as follows. Section~\ref{sec-II} is dedicated to the definition of the competitive spectral radius and the presentation of background results. In section~\ref{sec-III}, we study the continuity of the competitive spectral radius while in section~\ref{sec-IV}, we present a numerical algorithm to approximate the competitive spectral radius. Finally, in section~\ref{sec-V},
we discuss the population dynamics model and present numerical results.

\subsection{Related works}
The notion of matrix multiplication game was originally introduced by Asarin et al.~\cite{asarin_entropy}, who addressed the ``rectangular'' case, and
raised the question of the computability of the value in general:
``[the] corresponding matrix games no longer have a simple structure
(independent row uncertainty), and we conjecture that analysis of such games is non-computable''.
We answer here this question for families of {\em positive} matrices
that are non-rectangular (i.e.\ without any independent row uncertainty assumption). We show that in this case, the competitive spectral radius is approximable

As a generalization of the joint spectral radius, several properties of the competitive spectral radius stem from the properties of the joint spectral radius (JSR) and its counterpart, the lower spectral radius.
Heil and Strang deduced in~\cite{Heil1995} the continuity of the JSR from the Berger-Wang formula. Epperlein and Wirth showed that the  JSR is H\"older continuous (\cite{Wirth2023}).
Jungers proved
in~\cite[Prop.~1.11]{Jungers_2009}  the upper semi-continuity of the lower spectral radius (under the older name of ``joint spectral subradius''), and gave an example in which it is discontinuous. 
Bochi and Morris showed in~\cite{Bochi2014} that the lower spectral radius of family of invertible matrices is continuous under a \emph{1-domination} or \emph{multicone} condition (see also~\cite{BochiGourmelon,Avila2010}), 
Finally, Breuillard and Fujiwara have recently extended the JSR to families of isometries (\cite{Breuillard2021}). 

Even though the computation  of the JSR is generally undecidable~\cite{Blondel_2000}, several algorithms have been developed for approximating it, see in particular the works by Gripenberg~\cite{Gripenberg_1996}, Ahmadi, Jungers, Parrilo and Roozbehani~\cite{ahmadi}, Protasov~\cite{Protasov2021}, Guglielmi and Protasov~\cite{Guglielmi2011}. In contrast, the problem of computing the lower spectral radius is generally ill-posed (owing in particular to the lack of continuity properties), hence
most of the approaches focus on nonnegative sets of matrices (\cite{Guglielmi2011,Guglielmi2024}).
This explains why we require the operators to preserve a ``small cone'' $K$
in section~\ref{sec-IV}.  In the case of the positive orthant, the small
cone assumption is equivalent to the {\em tubular} condition of
Bousch and Mairesse~\cite{Bousch_2001}. 
Finally,
the combination of relative value iteration and Krasnoselskii-Mann damping
used in section~\ref{sec-IV} is motivated by the works~\cite{stott2020,akianmfcs}.

\section{The competitive spectral radius}\label{sec-II}
We now recall the definitions and main properties of the competitive spectral radius,
refering the reader to~\cite{AGM2024} for more information and proofs.

\subsection{Hemi-metric spaces}
We shall need the following notion of asymmetric metric~\cite{GAUBERT_2011}. 
\begin{Definition}[Hemi-metric]
	A map $d:X\times X \to \R$ is a \emph{hemi-metric} on the set $X$ if it satisfies the two following conditions:
	\begin{enumerate}
		\item $\forall(x,y,z) \in X^3, \; d(x,z) \leq d(x,y) + d(y,z)$ (triangular inequality)
		\item $\forall (x,y) \in X^2, \; d(x,y) = d(y,x) = 0$ if and only if $x=y$
	\end{enumerate}
\end{Definition}

Other definitions require $d$ to be nonnegative~\cite{Papadopoulos2008}, but this is undesirable in our approach. A basic example of hemi-metric on $\R^n$ is the map $\delta_n(x,y)\coloneqq \max_{i\in[n]} (x_i-y_i)$.

Given a hemi-metric $d$, the map
\[ d^{\circ}(x,y) = \max(d(x,y), d(y,x)) \enspace,
\]
is always a metric on $X$. 
We will refer to it as the \emph{symmetrized metric} of $d$.
In the sequel, we equip $X$ with the topology induced by $d^{\circ}$.

Given two hemi-metric spaces $(X,d)$ and $(Y,d')$, we introduce the notion of a
{\em $1$-Lipschitz} or {\em nonexpansive} function $\eigenvect$ with respect to $d$ and $d'$. Such a function verifies $d'(\eigenvect(x), \eigenvect(y)) \leq d(x,y)$ for all $x,y \in X$.

\subsection{The Escape Rate Game}

The \emph{Escape Rate Game} is the following two-player deterministic perfect information game. 

We fix two non-empty compact sets $\mathcal{A}$ and $\mathcal{B}$, that will represent the action spaces of the two players. We also fix a hemi-metric space
$(X,d)$, which will play the role of the state space. To each pair
$(a,b) \in \mathcal{A}\times \mathcal{B}$, we associate a nonexpansive
self-map $T_{ab}$ of $(X,d)$. The initial state $x_0 \in X$ is given. We construct inductively a sequence of states
$(x_k)_{k\geq 0}$ as follows. Both players observe the current state $x_k$.
At each turn $k\geq 1$, player \Min chooses an action $a_k \in \mathcal{A}$.
Then, after having observed the action $a_k$,
player \Max chooses an action $b_k \in \mathcal{B}$.
The next state is given by $x_{k}= T_{a_{k} b_{k}}(x_{k-1})$.

The {\em payoff in horizon $k$} is given by:
\begin{equation}
	\payoff_k(a_1 b_1 \dots a_k b_k) = d(T_{a_k b_k} \circ \dots \circ T_{a_1 b_1}(x_0),x_0) \enspace .\label{e-def-jk}
\end{equation}
We now define the \emph{escape rate game}, denoted $\Gamma$. An infinite number of turns are played, and player \Min wishes to minimize the {\em escape rate} or mean-payoff per time unit:
\begin{equation*}
\payoff_\infty(a_1b_1a_2b_2\dots)\coloneqq	\limsup_{k \to \infty}  \frac{\payoff_k (a_1 b_1 \dots a_k b_k)}k \enspace .\label{e-limsup}
\end{equation*} 
while player \Max wants to maximize it. It follows from
the nonexpansiveness of the maps $T_{ab}$ that this limit is independent of the choice of $x_0$.

A {\em strategy} is a map that assigns a move to any finite history of states and actions (for a sequence of turns). We refer the reader to~\cite{AGM2024} for more information.
The set of strategies of \Min (resp \Max) will be written $\strategyMin$ (resp $\strategyMax$). %

For every pair $(\sigma,\tau)\in \strategyMin\times\strategyMax$, 
we denote by $\payoff_k(\sigma,\tau)$ the payoff in horizon $k$,
as per~\eqref{e-def-jk},
assuming that the sequence of actions $a_1,b_1,a_2,b_2,\dots$ is generated according to the strategies $\sigma$ and $\tau$, and similarly, we denote by
$\payoff_\infty(\sigma,\tau)$ the payoff of $\Gamma$, that is the escape rate defined above.

Recall that the {\em value} of a game is the unique quantity that the players can both guarantee. Formally
a game with strategy spaces $\strategyMin$ and $\strategyMax$ and payoff
function $\strategyMin \times \strategyMax \to \R$, $(\sigma,\tau)\mapsto
\payoff(\sigma,\tau)$, {\em has a value} $\lambda \in \R$ if $\forall \epsilon > 0, \; \exists \sigma^*\in\strategyMin, \tau^*\in\strategyMax$ such that
\begin{equation*}
\payoff(\sigma^*, \tau) \leq \lambda + \epsilon \text{ and }\payoff(\sigma, \tau^*) \geq \lambda - \epsilon 
\end{equation*}
for all strategies $\sigma\in\strategyMin$ and $\tau\in\strategyMax$. Such strategies are called $\epsilon$-optimal strategies. When $\epsilon=0$,
one gets \emph{optimal strategies}, and in particular, $\lambda= \payoff(\sigma^*, \tau^*)$.

\subsection{The example of cones and the Funk metric}
In this paper we will focus on escape rate games where the hemi-metric space considered is the interior of a closed, convex and pointed cone $C \subset \R^n$ equipped with the Funk metric. This hemi-metric is defined as
\[
\Funk(x,y) = \log \inf \{\lambda > 0 \mid x \leq_C \lambda y \}\enspace ,
\]
where $\leq_C$ is the natural partial order on the cone ($x\leq_C y$ if $y-x\in C$).
Note that $\Funk(x,y)$ can be defined for all $(x,y) \in C\times \Int C$, but for the symmetrized Funk metric of $x$ and $y$ to exists we need $x$ and $y$ to be comparable, i.e.\ $\exists \mu, \lambda>0$ such that $\mu x \leq y \leq \lambda x$. The subsets of comparable elements form a partition of $C$ and are called the \emph{parts} of the cone $C$. We mostly consider the interior of the cone as it is often the most interesting part. The symmetrization of the Funk metric is called the Thompson or Thompson's part metric, and it defines on each part the same topology as the one induced by any norm on $\R^n$.
The Funk metric is closely related to the Hilbert's projective metric defined as
\[
\Hil(x,y)= \Funk(x,y) + \Funk(y,x)
\]
If $T$ is a self-map of a part $P$ of $C$, that satisfies $x\leq_C y\Rightarrow T(x)\leq_C T(y)$ and $T(\alpha x) = \alpha T(x)$
for all $x,y\in P$ and $\alpha>0$, it is immediate to check
that $T$ is non-expansive un the Funk hemi-metric.
It follows that it is also nonexpansive in the Thompson and Hilbert metric.\todo{SG: added this -- this was used but not said}
\begin{example}
Consider the nonnegative orthant $C=\R_{\geq 0}^n$. In this case, for $x,y\in (\R_{>0})^n = \Int C$, we have $\Funk(x,y) =\log \max_{i\in[n]} x_i/y_i$. Then, we take as actions spaces two sets $\A$ and $\B$ of matrices
that preserve the positive orthant. Every pair $(A,B)\in\A\times B$ determines
an operator:
\[
T_{AB}(x)=xAB\enspace .
\]
Taking $x_0=e$ (the unit vector), and setting $x_{k}=x_{k-1}A_kB_k=T_{A_kB_k}(x_k)$,
we see that
\[ J_k(A_1B_1\dots A_kB_k) = \Funk(x_k,x_0)=\log  \max_{i\in[n]} (x_k)_i
\enspace.
\]
Hence, we retrieve the matrix multiplication games model of~\cite{asarin_entropy}, in the case of nonnegative matrices (preserving the interior of the orthant).
\end{example}
\begin{example}
  Suppose now that $\A$ and $\B$ are sets of $n\times n$ invertible matrices.
  We now take for $C$ the cone $S_n^+$ of positive semidefinite matrices.
  Then, for all $X,Y\in \Int S_n^+$, we have $\Funk(X,Y)= \log\lambda_{\max}(XY^{-1})$. 
  To every pair $(A,B)\in\A\times \B$, we associate the congruence  operator
\[
T_{AB}(x)= (AB)^*x(AB)
\]
which preserves $\Int S_n^+$. Setting $x_0=I$ (the identity matrix)
and $x_k = T_{A_kB_k}(x_{k-1})$, we get
\begin{align*}
J_k(A_1B_1\dots A_kB_k) &= \Funk(x_k,x_0)\\&=2 \log  \sigma_{\max}(A_1B_1\dots A_kB_k)
\enspace,
\end{align*}
where $\sigma_{\max}$ denotes the maximal singular value.
In this way, we recover matrix multiplication games with invertible matrices.
\end{example}
\subsection{Characterizations of the competitive spectral radius}
{\em The proofs of the results in this section can be found in~\cite{AGM2024}.}

Given a hemi-metric space $(X,d)$, we denote by $\Lip_1(X)$, the space of functions from $X$ to $\R$ which are 1-Lipschitz with respect to $d$ and the hemi-metric on $\R$, $\delta_1: (x,y) \mapsto x-y$. We introduce the \emph{Shapley operator} $S$ of an escape rate game $\Gamma((X,d),(T_ab),\A,\B)$. It sends $\Lip_1$ to itself, as follows:
\[
Sv(x) \coloneqq \inf_{a \in \A} \sup_{b \in \B}v(T_{ab}(x))\enspace .
\]

Generally, a Shapley operator $S$ is a map from $\R^X$ to $\R^X$, which is additively homogeneous, i.e $S (\lambda e +v)= \lambda e + Sv$ and monotone with respect to the natural partial order on $\R^X$. If $(\lambda, v) \in \R \times \R^X$ verify $Sv = \lambda + v$ (Resp. $Sv \leq \lambda + v$, $Sv \geq \lambda + v$, they will be called additive eigenvalue and eigenvector (Resp. additive sub-eigenvalue, additive super-eigenvalue).

We introduce the sequence $(s_k)_{k \in \N}$ defined as follows
\[
s_k \coloneqq [S^k d(\cdot,x_0)](x_0)
\]
It is subadditive and therefore Fekete's subadditive lemma tells us that 
\[
\lim_{k \to \infty} \frac{s_k}{k} = \inf_{k \in \N} \frac{s_k}{k} \in \R \enspace .
\]
    {\em In the sequel, we will consider families of nonexpansive operators satisfying the following assumption.}
\begin{hypo}\label{assump-1}
For all $x\in X$,  
the maps $b \mapsto T_{ab}(x)$ and $a \mapsto T_{ab}(x)$ are continuous,
  and for all compact sets $K$, the set $\{T_{ab}(x) \ | \ (a,b,x) \in \mathcal{A} \times \mathcal{B} \times K\}$ is compact, 
\end{hypo}
 We next state a general theorem about the existence and characterization of the value of an escape rate game. 
\begin{Theorem}\label{main}
  The escape Rate Game has a value $\rho$ given by
	\begin{align*}
	\rho &= \lim_{k\to\infty}\frac{[S^k d(\cdot,\bar{x})](\bar x)}{k} \\
       &= \max\{ \lambda \in \R \mid \exists\; v \in \Lip_1, \lambda + v \leq Sv \}\\
       &= \max_{v \in \Lip}\inf_{x \in X}(Sv(x)-v(x)) \enspace .
  \end{align*}
\end{Theorem}
    
For cones equipped with the Funk Metric, we can obtain a dual characterization in terms of distance-like function.
\begin{Definition}
A $1$-Lipschitz function $v: X\to \R$ is {\em distance-like} if there exist
$x_0\in X$ and a constant $\alpha\in \R$ such that\
\(
v(x) \geq \alpha+ d(x,x_0), \quad \forall x\in  X
\).
\end{Definition}

\begin{Proposition}
	\label{dlike}
  Suppose that there exists a distance-like function $v$ such that
  \(
  S v\leq \lambda + v \),
  for some $\lambda\in \R$. Then, the value $\rho$ of the escape rate game satisfies $\rho \leq \lambda$.
\end{Proposition}

On a part of a cone $C$, the topologies induced by the Funk metric and a norm are the same (see \cite{nussbaumlemmens} corollary 2.5.6). So a continuous map for the Funk metric will also be continuous with respect to the euclidean topology. The issue will be at the boundary of the cone where the two topologies differ, but as long as the operators $(T_{ab})$ can be continously extended to the whole cone, we get the following theorem.

\begin{Theorem}[Dual Characterization] \label{th-dual}
  Suppose that every operator $T_{ab}$ extends continuously to the closed cone $C$, with respect to the Euclidean topology. 
		The value $\rho$ of the escape rate games played on the interior of $C$ has the following dual characterization
		\begin{align}\label{e-dual}
		\rho = \inf\{\lambda \in \R \mid \exists v \in \mathscr{D}_C, Sv \leq \lambda + v \}
		\end{align}
where $\mathscr{D}_C$ is the subset of continuous (for the euclidean topology) function on $C$ which are 1-Lipschitz on $\Int C$. 
\end{Theorem}

This result stems from the properties of nonexpansive operators for the Funk metric. They are positively homogeneous of degree 1, i.e for all $(\lambda,x) \in \R_{>0} \times \Int C, T_{ab}(\lambda x) = \lambda T_{ab}(x)$. Moreover 1-Lipschitz functions are \emph{log-homogeneous}, i.e.\ if $v \in \Lip_1(\Int C)$, for all $(\lambda,x) \in \R_{>0} \times \Int C, v(\lambda x) = \log(\lambda) + v(x)$, and so is the Funk metric in its first variable.

It will be convenient, in~\Cref{sec-IV},
to solve a ``normalized'' version of the eigenproblem~\eqref{e-dual}, restricting the function $v$ to a cross-section of the cone $C$.
We set $\Delta \coloneqq \{x \in C \mid \langle x,e^* \rangle = 1 \}$, where $e^*$ is a point in the interior of the dual cone $C^*$.
We consider the space $\CLip(\Delta) \coloneqq \mathscr{C}(\Delta)\cap \Lip_1(\relint\Delta)$ equipped with the hemi-metric $\delta : (f,g) \mapsto \max_{x \in \Delta}(f(x) - g(x))$ (see \cite{GAUBERT_2011}, \cite{AGM2024} for more details).

We define the operator $\shapc$, sending a function $v\in \CLip(\Delta)$
to the function 
\[
\shapc v(x) \coloneqq \inf_{a \in \mathcal{A}}\sup_{b \in \mathcal{B}}\log(\langle T_{ab}(x),e^*\rangle) + v\Big(\frac{T_{ab}(x)}{\langle T_{ab}(x),e^*\rangle}\Big) \enspace.
\]
\begin{remark}
Notice that if we evaluate $\shapc v$ at a function $v$ belonging to $\Lip_1(\Int C)$, we get exactly $Sv$. 
\end{remark}
\begin{proposition}
  The operator $\shapc$ preserves $\CLip(\Delta)$.
\end{proposition}
Every function $v \in \CLip(\Delta)$ extends
to a function $\bar{v} \in \mathscr{D}_C$ by setting
$\bar v(\lambda x) = \log(\lambda) + v(x)$ for all $\lambda>0$ and $x\in \Delta$. We verify that $\shapc v\leq \lambda +v$ iff $S\bar{v} \leq \lambda +\bar{v}$.
Conversely, every function $w\in \mathscr{D}_C$ satisfying $Sw\leq \lambda +w$
restricts to a function $v\in \CLip(\Delta)$ satisfying $\shapc v\leq \lambda +v$.
Note that $w$ is determined uniquely by its restriction $v$. Therefore,
the eigenproblem for $\shapc $ is a reformulation of the eigenproblem for $S$.
This reduction will prove to be very useful to compute numerically the competitive spectral radius.

\section{Continuity of the competitive spectral radius}\label{sec-III}

We now explore the continuity of the competitive spectral radius with respect to the family of operators $T_{a,b}$ and to the actions spaces $\A$ and $\B$. 

We suppose that the action spaces are compact subsets of a metric space $(E,\delta)$. We equip the space of compact subsets of $E$ the Hausdorff distance, denoted $\hausdorff$.
Given a family
of nonexpansive self-maps of $X$, $T=(T_{a,b})_{(a, b)\in E\times E}$,
and compact subsets $\A\subset E$ and $\B\subset E$, we denote
by $\rho(T,\A,\B)$ the value of the associated game.
The following 
theorem provides a condition for $\rho$ to be Lipschitz continuous as a function
of the action spaces and of the operator.
\begin{Theorem}\label{th-continuity}
  Let $T=(T_{a,b})_{(a,b)\in E\times E}$ and
  $T'=(T'_{a,b})_{(a,b)\in E\times E}$ be two families
  of maps such that
  there exists $\epsilon \geq 0$
and $L>0$ such that, 
  for all
  $(a, b), (a',b') \in E\times E$,
  and for all $x\in X$,
    \[
    d^{\circ}(T_{ab}(x), T'_{a' b'}(x))\leq L(\delta(a,a') + \delta(b,b')) + \epsilon
    \]
    Then, for all compact subsets $\A,\A'\subset E$
      and $\B,\B'\subset E$,
      \[
| \rho(T,\A,\B)-\rho(T',\A',\B')|
 \leq  L (\delta_H(\A,\A')+ \delta_H(\B,\B')) +\epsilon
 \]
\end{Theorem}
\begin{proof}
  We denote by $S$ (resp.\ $S'$) the Shapley operator associated to
  $(T,\A,\B)$
  (resp.\ $(T',\A',\B')$).
  By~\Cref{main},
  there exists $v \in \Lip_1$ such that $S v \geq \rho(T,\mathcal{A},\mathcal{B}) + v$. Because $v$ is $1$-Lipschitz, for all $a,a'\in E$ and $b,b'\in E$, for all $x\in X$,
 we have 
\[
v(T_{ab}(x)) \leq v(T'_{a' b'}(x)) + d(T_{ab}(x), T'_{a' b'}(x)) \enspace .
\]
So for all $(a,b,a') \in \mathcal{A}\times \mathcal{B}\times \mathcal{A}'$ we have 
\[
v(T_{ab}(x)) \leq \inf_{b' \in \mathcal{B}'} \left[v(T'_{a' b'}(x)) + d(T_{ab}(x), T'_{a' b'}(x))\right]
\]
Therefore, for all $a' \in \mathcal{A}'$
\begin{align*}
    &Sv(x)\\
    &\leq \inf_{a \in \mathcal{A}}\sup_{b\in \mathcal{B}}\big[\inf_{b' \in \mathcal{B}'} v(T'_{a' b'}(x)) + d(T_{ab}(x), T'_{a' b'}(x))\big]\\
    &\leq \inf_{a \in \mathcal{A}}\sup_{b\in \mathcal{B}}\big[\sup_{b' \in \mathcal{B}'} v(T'_{a' b'}(x)) + \inf_{b' \in \mathcal{B}'}d(T_{ab}(x), T'_{a' b'}(x))\big]\\
    &= \sup_{b' \in \mathcal{B}'} v(T'_{a' b'}(x)) + \inf_{a \in \mathcal{A}}\sup_{b\in \mathcal{B}} \inf_{b' \in \mathcal{B}'}d(T_{ab}(x), T'_{a' b'}(x))
\end{align*}
By assumption, for all $a' \in \mathcal{A}'$
\begin{align*}
    &\inf_{a \in \mathcal{A}}\sup_{b\in \mathcal{B}} \inf_{b' \in \mathcal{B}'}d(T_{ab}(x), T'_{a' b'}(x)) \\
  &\leq \inf_{a \in \mathcal{A}}\sup_{b\in \mathcal{B}} \inf_{b' \in \mathcal{B}'}
  \big( L\delta(a,a')+L\delta(b,b') + \epsilon \big) ,
  \text{ and so}
\end{align*}
\begin{align*}
Sv(x)&
\leq \inf_{a' \in \mathcal{A}'}\big[\sup_{b' \in \mathcal{B}'} v(T'_{a'b'}(x)) + \\
  & \qquad \qquad
  \inf_{a \in \mathcal{A}}\sup_{b\in \mathcal{B}} \inf_{b' \in \mathcal{B}'}(L\delta(a,a')+L\delta(b,b'))\big] + \epsilon \\
&
\!\!\!\!\!\!\!\!\!\!\!\leq \inf_{a' \in \mathcal{A}'}\sup_{b' \in \mathcal{B}'} v(T'_{a' b'}(x)) + \\
  & \qquad 
\sup_{a' \in \mathcal{A}'}\inf_{a \in \mathcal{A}}\sup_{b\in \mathcal{B}} \inf_{b' \in \mathcal{B}'}(L\delta(a,a')+L\delta(b,b'))+\epsilon \\
    &
\!\!\!\!\!\!\!\!\!\!\!= S' v(x) + L(\sup_{a' \in \mathcal{A}'}\inf_{a \in \mathcal{A}}\delta(a,a') + \sup_{b \in \mathcal{B}}\inf_{b' \in \mathcal{B}'}\delta(b,b')) + \epsilon\\
    &
\!\!\!\!\!\!\!\!\!\!\!\leq S' v(x) + L(\hausdorff(\mathcal{A},\mathcal{A}')+\hausdorff(\mathcal{B},\mathcal{B}')) + \epsilon
\enspace .
\end{align*}
Since this holds for all $x$, we have
\[
    \rho(T,\mathcal{A},\mathcal{B}) - L(\hausdorff(\mathcal{A},\mathcal{A}')+\hausdorff(\mathcal{B},\mathcal{B}')) - \epsilon + v \leq S' v \enspace .
\]
By~\Cref{main} (second equation), this implies that 
\[
    \rho(T,\mathcal{A},\mathcal{B}) - L(\hausdorff(\mathcal{A},\mathcal{A}')+\hausdorff(\mathcal{B},\mathcal{B}')) - \epsilon \leq \rho(T',\mathcal{A}', \mathcal{B}')
    \]
    By symmetry, we get the final inequality of the theorem.
\end{proof}
 
We now consider the following special case.
We denote by $C$ a closed convex pointed 
cone in $\R^n$, and by $\End(\Int C)$ the set
of linear operators preserving the interior of $C$.
The set $\End(\Int C)$ is equipped with the partial order $\leq$
such that for $A,A'\in\End(\Int C)$,  we have
$A\leq A'$ if $xA\leq_{C} xA'$ for all $x\in \Int C$
(the linear operators $A,A'$ act at right on the row vector $x$ by matrix
multiplication).
In this way, the notion of {\em part}
carries over to $\End(\Int C)$.
The Funk and Thompson metrics are well-defined on every part and still denoted by $\Funk(\cdot,\cdot)$ and $d_T(\cdot,\cdot)$.
For $(A,B)\in \End(\Int C)\times \End(\Int C)$, and $x\in \Int C$,
we set $T_{AB}(x) = xAB$.
For all $\A,\B\subset\End(\Int C)$,
we denote $\rho(\A,\B)\coloneqq \rho(T,\A,\B)$.
\begin{corollary}[Continuity of the competitive spectral radius, case of cones]\label{cor-continuity}
  The map which sends a pair $(\A,\B)$ of compact sets included in a part
$\mathscr{P}$  of $\End(\Int C)$,
  equipped with the Hausdorff distance induced by the Thompson metric,
  to $\rho(\A,\B)$ is $1$-Lipschitz.
\end{corollary}
\begin{proof}
  For $A,A' \in \End(\Int C)$, we have
  \(
  \Funk(A,A')= \log \inf \{\lambda > 0 \mid xA \leq \lambda x A', \forall x \in \Int C \}\).
  So, for all $x$ in $\Int C$, $\Funk(xA,xA') \leq \Funk(A, A')$. Furthermore, for all $B,B' \in \End(\Int C)$ as $(\Int C)A' \subset \Int C$, $\Funk(xA'B, xA'B') \leq \Funk(B,B')$.
  So, for all $x$ in $\Int C$,
\(
\Funk(xAB,xA'B')
\leq\Funk(xAB, xA'B) +
\Funk(xA'B,xA'B')
\leq\Funk(xA,xA') + \Funk(B,B')
\leq\Funk(A,A') + \Funk(B,B')
\leq d_T (A,A') + d_T(B,B')\).
  By symmetry we get that 
  \(
  d_T(xAB,xA'B') \leq d_T(A,A') + d_T(B,B')\).
  So the assumption of \cref{th-continuity} is verified with $\epsilon=0$ and $L=1$, and we get that $(\A,\B) \to \rho(\A,\B)$ is 1-Lipschitz with respect to the Hausdorff distance induced by the Thompson metric on the space of
  compact subsets of $\mathscr{P}\times \mathscr{P}$.
\end{proof}
\begin{example}\label{ex-support}
  Each subset $U\subset \{0,1\}^{n\times n}$
  such that for all $i\in [n]$, there exists $j\in [n]$ such that $(i,j)\in U$
  defines the set of matrices
  \(
  \mathscr{P}_U = \{ A\in \R_{\geq 0}^{n\times n}\mid A_{ij}>0 \iff (i,j) \in U\}\).
  We identify a matrix $A\in \mathscr{P}_U$ to the operator $x\mapsto xA$.
It is immediate that any set $\mathscr{P}_U$ is a part
  of $\End(\Int \R_{\geq 0}^n)$. 
Moreover, every part is of this form.
  For instance, taking $C=\R_{\geq}^2$, 
  $A=(1,1;1,1)$ and $A'=(1,1;0,1)$, we see that the linear maps $x\mapsto xA$ and $x\mapsto xA'$ belong
  to different parts of $\End (\Int \R_{\geq}^2)$.
\end{example}
\begin{remark}
  Let $C=\R_{\geq}^n$, let $\varnothing \neq U\subset [n]\times [n]$.
  One can show
  that on the set of compact subsets of $\mathscr{P}_U$, the Hausdorff
  distances defined by the Thompson metric and by the Euclidean metric induce
  the same topology. (The proof builds on principles similar to
  the ones of \cite{nussbaumlemmens} corollary 2.5.6). Therefore we also retrieve the continuity of the lower spectral radius for positive matrices, proven in~\cite{Jungers2012}. 
\end{remark}
\section{Approximating the competitive spectral radius}\label{sec-IV}

Let $C \subset \R^n$ be a pointed, convex and closed cone,
with dual cone $C^*$.
We choose $e^* \in \operatorname{int}C^*$ and denote by $\Delta=\{x\in C\mid e^*(x)=1\}$ the associated ``simplex''. We shall make the following assumption.
\begin{hypo}[Small cone]\label{small-cone}
There exists a closed cone 
$K \subset C$ such that $T_{ab}(K)\subset K$ for all $a\in \A,\; b\in \B$,
and $X\coloneqq K\cap \Delta \subset \operatorname{relint}\Delta$.
\end{hypo}
Since $\Delta$ is compact in the norm topology of $\R^n$,
so is $X$. Moreover,
on $\relint \Delta$, the norm topology coincides with the topology
induced by the Hilbert metric, so $X$ is compact as well in this
topology.
Hence, for all $h>0$, we can ``discretize'' $X$ by choosing
a finite subset $X_h\subset X$, of cardinality
$N_h$, such that
\(
X\subset \cup_{x\in X_h} B_H(x, h)
\),
where $B_H(x, h)$ is the open ball of center $x$ and radius $h$
in Hilbert metric.

Our goal is to compute the competitive spectral radius $\rho$ of the family
of operators $T_{ab}$. Following section~\ref{sec-II},
we will work with the following operator $F : \Lip_1(X) \to \Lip_1(X)$
\[
Fv(x) = \inf_a \sup_b \log(\langle T_{ab}(x), e^* \rangle) + v\Big(\frac{T_{ab}(x)}{\langle T_{ab}(x),e^* \rangle}\Big)\enspace .
\]

We define the {\em interpolation operator} $I^+_h: \R^{X_h} \to \R^X$ as follows
\[
I^+_h v(x) \coloneqq \min_{y \in X_h}[v(y) + \Funk(x,y)]
\]
and the restriction operator $R_h: \R^X \to \R^{X_h}$,
such that $R_h v(x)=v(x)$ for all $x\in X_h$.
The introduction of $I^+_h$ is motivated by the following property.
\begin{prop}[$\Lip_1$ extension]\label{prop-ext}
  For all $v\in \R^{X_h}$, we have $I^+_hv\in \Lip_1(X)$.
  Moreover,
  if
  $v: X_h \to \R$ is in $\Lip_1(X_h)$ then $R_h(I_h^+v) = v$.
\end{prop}
\begin{proof}
  Because of the triangular inequality, for all $x_i \in X_h$, $x \mapsto v(x_i) + \Funk(x,x_i)$ is 1-Lipschitz so $I_h^+ v \in \Lip_1$.
  Moreover, if $v\in\Lip_1(X_h)$, for all $x_i,x_j \in X_h$ , we have
  $v(x_i) \leq v(x_j) + \Funk(x_i,x_j)$. Taking the infimum
  over $x_j\in X_j$, we deduce that $v(x_i)\leq I_h^+ v(x_i)$.
Moreover, considering $x_j=x_i$, we get $I_h^+ v(x_i)\leq v(x_i)$.
\end{proof}
\begin{remark}
Alternatively, we may consider the dual interpolation operator
  \(
  I^-_h v(x) \coloneqq \max_{y \in X_h}[v(y) - \Funk(y,x)]\).
Proposition~\ref{prop-ext}, and the results which follow, admit dual versions, working with $I^-_h$. 
The operators $I^\pm_h$ are examples of McShane-Whitney extensions~\cite{McShane1934}, \cite{Whitney1934}.
  \end{remark}

Define the Shapley operators:
\[
F_h^+ = F I_h^+ R_h : \Lip_1(X)\to\Lip_1(X)
\]
and
\[
\hat{F}_h^+ = R_h F I_h^+: \Lip_1(X_h) \to \Lip_1(X_h)\enspace .
\]
We shall use the following general property. Here,
$E$ is an arbitrary hemi-metric space. 
\begin{Proposition}
  Suppose that $G$ is a Shapley operator sending $\Lip_1(E)$ to itself.
  Then, the additive eigenvalue $\lambda$ of $S$ associated to a  distance-like eigenvector (when it exists) satisfies
  \[
  \lambda =\specrad(G)\coloneqq \lim_{k \to \infty}\frac{[G^k d(\cdot,x_0)](x_0)}{k} \enspace.
  \]
  In particular, $\lambda$ is unique.
\end{Proposition}
\begin{proof}
  Let $(\lambda,v) \in \R \times \Lip_1(E)$ such that $Gv = \lambda +v$ where $v$ is distance-like. Then, for all $k\in\N$, $G^kv = k\lambda +v$, and so $\lambda = \lim_{k \to \infty}\frac{G^k v(x_0)}{k}$.
  Moreover, there is a constant $\alpha$ such that $\alpha + d(\cdot,x_0)\leq v$, and so
\(
\alpha + G^kd(\cdot,x_0) \leq G^kv \leq v(x_0) + G^kd(\cdot,x_0) 
\).
So $\lambda = \lim_{k \to \infty}\frac{G^k v(x_0)}{k} = \lim_{k \to \infty}\frac{[G^k d(\cdot,x_0)](x_0)}{k}=\specrad(G)$.
\end{proof}

\begin{theorem}\label{th-fp}
  Let $(E,d)$ denote a hemi-metric space, and let $G$ be a Shapley operator
  preserving $\Lip_1(E)$ and continuous for the topology of uniform
  convergence on compact sets.
  Then, there exists a vector $v\in \Lip_1$ and
  a scalar $\lambda\in \R$ such that $G(v)=\lambda +v$. 
\end{theorem}
This follows from the Schauder-Tychonoff's fixed-point theorem.
The equation $G(v)=\lambda +v$ is an instance
of the ``ergodic eigenproblem'' which has a long history in the
ergodic optimization litterature~\cite[\S~2.5]{maslovkolokoltsov95}, \cite{Savchenko1999,Bousch2001,Jenkinson2018}. \todo{SG: I wrote here that this is a more or less standard result (rather than in intro)}
\begin{corollary}\label{cor-S}
  Each of the operators $F$,  $F_h^+$ and $\hat{F}_h^+$ admits a distance-like
  eigenvector.
\end{corollary}
\begin{proof}
  The Shapley operator $F$ preserves $\Lip_1(X)$.
  Since $X$ is compact, we can define the sup-norm
  of every function in $\Lip_1(X)$. The Shapley operator
  $F$ is non-expansive in this sup-norm.
  Hence, $F$ satisfies the continuity assumption of~\Cref{th-fp}.
  It follows that there exists $v\in \Lip_1(X)$ and $\lambda\in \R$
  such that $Fv=\lambda +v$. Since $X$ is compact,
  the function $v$ which is bounded on $X$ is trivially distance-like.
  The same argument applies to $F_h^+$ and to $\hat{F}_h^+$.
  \end{proof}
This allows us to define the numbers
$\specrad(F)$, $\specrad(F_h^+)$,
and $\specrad (\hat{F}_h^+)$. The non-linear
eigenproblem for the operator $F$ is a two-player
version of the cohomological equation arising in the study
of dynamical systems~\cite{Livic1972}.

\begin{proposition}
  We have 
      $\specrad(\hat{F}_h^+)=\specrad(F_h^+)$.
\end{proposition}

\begin{proof}
  Let $v \in \Lip_1(X), \lambda \in \R$ such that $F_h^+v = \lambda + v$,
  where $\lambda =\specrad(F_h^+)$.
  Left-composing $F_h I_h^+ R_h v = \lambda +v $ by $R_h$, and setting
  $\hat{v}\coloneqq R_h v$, we get $\hat{F}_h^+ \hat{v} = \lambda + \hat{v}$.
  By uniqueness of the eigenvalue of $\hat{F}_h^+$, we get
  $\specrad(\hat{F}_h^+)= \specrad(F_h^+)$.
\end{proof}
The above eigenvalue gives an approximation of the value of the game $\rho$, whose precision only depends on the mesh-size of the grid $X_h$. By mesh-size of the grid, we means $h$ here.

\begin{theorem}\label{th-effective}
  Let $v \in \Lip_1(X_h)$ and $\lambda \in \R$ such that $\hat{F}_h^+v = \lambda + v$ then
  \(
-h + \lambda  \leq \rho \leq \lambda\).
\end{theorem}
\begin{proof}
  Let us define $g \coloneqq I_h^+ v$. It is a 1-Lipschitz function. Let $x \in X$, as $Sg$ is 1-Lipschitz we get that for all $x_i \in X_h$ we have
\(
Sg(x) 
\leq Sg(x_i) + \Funk(x,x_i)
= R_h SI_h^+ v (x_i) + \Funk(x,x_i)
= \lambda + v(x_i) + \Funk(x,x_i)\).
So
\(
  Sg(x) \leq  \lambda + \min_{x_i \in X_h}[v(x_i) + \Funk(x,x_i)]
  = \lambda + I_h^+ v(x) = \lambda + g(x)
\). 
Now, by~\Cref{th-dual}, we have $\rho \leq \lambda$.

For the other inequality, we get from the nonexpansiveness of $Sg$ that for all $x_i \in X_h$
 \(
    Sg(x) 
    \geq Sg(x_i) - \Funk(x_i,x)
    = \lambda + v(x_i) + \Funk(x,x_i) - \Funk(x,x_i) - \Funk(x_i,x)
\). 
  So 
  \(
    Sg(x) \geq \lambda + g(x) - \min_{x_i \in X_h}\Hil(x,x_i) = \lambda + g(x) - h
  \),
  which tells us that $\rho \geq \lambda - h$
\end{proof}

Therefore, to approximate numerically the value of our game, it suffices
to compute the additive eigenvalue $\specrad(\hat{F}_h^+)$. Observe
that $\hat{F}_h^+$ acts on the finite dimensional set $\Lip_1(X_h)\subset \R^n$.

To this end, we shall combine
the techniques of {\em relative value iteration}~\cite{WHITE1963373}
and {\em Krasnoselskii-Mann damping}~\cite{mann,krasno} ,
following \cite{akianmfcs}.  
We fix arbitrarily one point $\bar x\in X_h$, and define the normalized operator
\[
\hat{F}_{hn}^+ v\coloneqq \hat{F}_h^+ -[\hat{F}_h^+v] (\bar x)
\]
$\hat{F}_{hn}$ preserves the subspace
$\mathscr{L}_{x_i}(X_h)= \{ v\in \Lip_1(X_h) \mid v(\bar x)=0\}$.
Moreover, since $\hat{F}^+_h$ is a Shapley operator, it
non-expansive in Hilbert's seminorm,
$\lVert v \rVert_H \coloneqq \max_j v(x_j) - \min_j v(x_j)$,
an so does the normalized operator $\hat{F}_{hn}^+$.
Observe that the restriction of Hilbert's seminorm to
$\mathscr{L}_{\bar x}$ is a norm.

Starting from an arbitrary vector $v_0\in \mathscr{L}(\bar x)$,
we define the sequences: 
\begin{align}
v_{k+1} &= \frac{\hat{F}_{hn}^+(v_k) + v_k}2 ,\qquad 
\lambda_{k+1} = [\hat{F}_{h}^+ v_{k+1}](\bar x) \enspace .\label{eq-rvikm}
\end{align}
This is formalized as \Cref{algo-rvikm}.

The next theorem follows from a result of Ishikawa~\cite{ishikawa},
concerning the convergence of the sequence $x_{k+1} = (T(x_k)+x_k)/2$,
for a nonexpansive operator $T$ sending a nonempty compact subset $D$ of a Banach space to itself, and from an error estimate of Baillon and Bruck~\cite{baillonbruck}.

\begin{theorem}[Corollary of~\cite{ishikawa,baillonbruck}]\label{th-bb}
  The sequence $(\lambda^k,v^k)$ converges
  to a pair $(\lambda,v)$ verifying $\hat{F}_{hn}^+ v = \lambda + v$. Moreover,
  for any eigenvector $w$ of $\hat{F}_{hn}^+$,
  \[
    \lVert v^{k+1} - v^k \rVert_H \leq \frac{2 \operatorname{dist}_H(v^0,w)}{\sqrt{\pi k}} \enspace .
    \]
\end{theorem}
\begin{algorithm}
  \caption{Relative value iteration with Krasnoselski-mann damping -- RVI-KM.}\label{algo-rvikm}
  Input: a desired precision $h$, and the operators $T_{a,b}$.
  Start with $v_0 =0 \in \mathscr{L}_{\bar x}(X_h), \lambda_0 = 0\in \R$.
  Compute the sequences defined by ~\eqref{eq-rvikm}, 
until $\|v_{k+1}-v_k\|_H\leq h$. Return $\lambda_k$.
\end{algorithm}
\begin{theorem}\label{th-algo}
  The RVI-KM algorithm stops after at most
  \[ k= \frac{4}{\pi}(\frac{M^+-M^-}{h})^2
  \]
  steps, where
  \begin{align*}
M^+ &\coloneqq \max_{a,b \in \A\times \B,x,y\in X}\Funk(T_{ab}x,y) \\
M^- &\coloneqq \min_{a,b \in \A\times \B,x,y\in X}\Funk(T_{ab}x,y)
\end{align*}
and the final value of $\lambda_k$ is such that $\rho \in [-3h + \lambda_k, 2h+\lambda_k]$.
 \end{theorem}
\begin{proof}
  If $(w,\lambda)$ is an eigenpair of $\hat{F}^+_{h}$, we
  have, for all $x\in X_h$,
\begin{align*}
\lambda + w(x) 
= \inf_{a}\sup_{b}\inf_{y \in X_h}w(y) + \Funk(T_{ab}(x),y) 
\enspace.
 \end{align*}
  Then, setting
\(
w^+ = \max_x w(x) \text{  and  } w^- = \min_x w(x),\)
we get $\lambda + w(x) \geq  M^- + w^-$ for all $x\in X_h$,
and choosing $x$ such that $w(x)=w^-$, we deduce that $\lambda \geq M^-$.
Similarly, we have $\lambda + w(x) \leq M^+ + w^-$, for all $x\in X_h$,
and so $\|w\|_H \le M^+-\lambda \leq M^+-M^-$. Then, the termination
result follows from~\Cref{th-bb}. Moreover, 
the algorithm returns $v_k$ and $\lambda_k$ such that $\delta_k\coloneqq
\hat{F}^+_h(v_k)-\lambda_k -v_k $ with $\lambda_k =[\hat{F}^+_h(v_k)](\bar x)$
satisfies $\|\delta_k\|_H\leq 2h$. Since $\delta_k(\bar x)=0$,
we deduce that $\max_x \delta_k(x) \leq \|\delta_k\|_H\leq 2h$, and
dually $\min_x \delta_k(x) \geq -2h$. It follows that
$-3h + \lambda_k + v_k \leq \hat{F}^+_h(v_k)
\leq 2h + \lambda_k + v_k$. Hence, $\specrad(\hat{F}^+_h)\in [\lambda_k-3h,\lambda_k+2h]$. Together with~\Cref{th-effective}, this gives $\rho \in [\lambda_k -3h, \lambda_k +2h]$.
\end{proof}
The exact evaluation of the operator $\hat{F}^+_{hn}$ requires finite action spaces. \Cref{cor-continuity} allows us to reduce to this case, after discretization
of the action spaces.
\begin{corollary}
  Let $\A$ and $\B$ be two compact subsets of
 the same part of $\End(\Int(C))$. Then for any $h>0$, there exists finite subsets $\A_f \subset A$ and $\B_f \subset B$ such that 
  \(
  \lvert \rho(\A,\B) - \rho(\A_f, \B_f) \rvert \leq h
  \)\hfill\qed
\end{corollary}
\begin{remark}
  Every iteration of the RVI-KM algorithm requires
  the computation of $[\hat{F}^+_h(v)](x)=[FI_h^+(v)](x)$ at all the points $x$ of the grid $X_h$. When the action spaces $\A$ and $\B$ are finite, this requires
  $O(|\A||\B| |X_h|^2)$ arithmetic operations. When the cone $C$ is of dimension
  $d$, the simplex $\Delta$ is of dimension $d-1$, and so $|X_h|= O(1/h^{d-1})$. Hence, the RVI-KM algorithm
  requires a total of $O(|\A||\B|/h^{2d})$ arithmetic operations.
\end{remark}

\section{Application: game of population dynamics}\label{sec-V}
We consider an age-structured population model with three ages,
leading to a discrete-time dynamics of the form
\[
x_{k+1} = x_k L(\alpha_k,\beta_k)  , \;
x_k \in \R_{>0}^{1\times 3}, \;
x_0 \textrm{ given},
\]
where $L(\alpha_k,\beta_k)$ is the (transposed) Leslie matrix:
\[
L(\alpha_k,\beta_k) = \begin{pmatrix}
  \beta^1_k & \alpha^1_k & 0\\
  \beta^2_k & 0 & \alpha^2_k\\
  \beta^3_k & 0 & 0
\end{pmatrix} \enspace.
\]
Here, $i\in \{1,2,3\}$, denotes the age,
$\beta^i_k>0$ denotes the number of children of an individual
of age $i$, at time $k$, and for $i\in\{1,2\}$, $\alpha^i_k\in (0,1)$ denotes the proportion
of individuals having age $i$ at time $k$, which are still alive at age $i+1$
and time $k+1$. We assume that at each time, \Min\ chooses the
vector $\alpha_k\coloneqq (\alpha^1_k,\alpha^2_k)$ within a prescribed compact
set $\A$ included in $\R_{>0}^2$,
whereas \Max\ chooses the vector $\beta_k\coloneqq (\beta^1_k,\beta^2_k,\beta^3_k)$
within another prescribed compact set
$\B$ included in $\R_{>0}^3$. We refer the reader to~\cite{Perthame2007}
for more information on age-structured models (in a broader
infinite dimensional setting).

This can be interpreted as a model of population of mosquitos with three different development stages. \Min wants to limit the growth of this populaton by influencing the survival rate $\alpha_k$,
whereas \Max (the population of insects) wants to maximize its growth, and can influence the reproduction rate $\beta_k$.

For $\beta\in \A$ and $\alpha\in \B$, and $x\in \R_{>0}^3$, we set
\[
T_{ab}(x)=x L(\alpha,\beta)
\]
The assumptions of \cref{main} are verified:
Leslie matrices preserve the interior
of the cone $C=\R_{>0}^3$; moreover, our assumption of compactness of $\A\subset \R_{>0}^2$ and
$\B_{>0}^3$ entail that the entries of $\alpha$ and $\beta$ are bounded away from zero,
which entails that all the matrices $L(\alpha,\beta)$ with $(\alpha,\beta)\in \A\times \B$
are included in a common part of $\End(\Int C)$, see~\Cref{ex-support}.

Specifically, we take the action spaces
\begin{align*}
  \A &\!\coloneqq\! \{ (0.9,0.6), (0.6,0.9),(0.7,0.7)\} \!=\! \{\alpha_1, \alpha_2, \alpha_3\}\\
  \B &\!\coloneqq\! \{(0.2,1.4,1.4), (0.2,1.7,1), (0.2,1,1.7)\} 
\!=\! \{\beta_1, \beta_2, \beta_3\}
\end{align*}

The RVI-KM algorithm was implemented with python 3.11.11 on a Macbook pro with an Apple M3 pro 11 core CPU and 18 GO of LPDDR5 ram.
\begin{figure}[h]
  \includegraphics[scale=0.65]{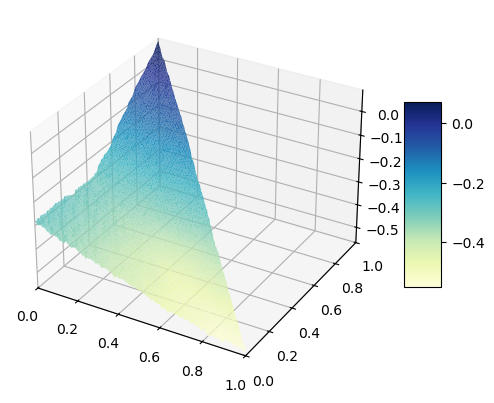}
  \caption{Eigenfunction computed with 8646 points of discretization}
\end{figure}

\begin{table}[h!]
  \centering
  \begin{tabular}{|c c c c|}
    \hline
    Discretization points & Value & Iterations & Runtime (s) \\
    \hline \hline
    3321 & 1.3158 & 10 & 4.23\\
    5151 & 1.3148 & 11 & 11.39\\
    8646 & 1.3147 & 12 & 52.42\\
    13041 & 1.3120 & 12 & 1584.62\\
    \hline
  \end{tabular}
  \caption{Benchmarks}
\end{table}
Starting from the center of the simplex, $x_0 = (1/3,1/3,1/3)$, we get the sequence of optimal moves
{\small\(
\alpha_3\beta_1\alpha_2\beta_1\alpha_2\beta_2\alpha_2\beta_1\alpha_2\beta_2\alpha_2\beta_2\alpha_2\beta_1\alpha_2\beta_1(\alpha_2\beta_2)^\omega\)},
where $(\cdot)^\omega$ denotes the periodic repetion of a string.
We notice a ``Turnpike property'': all optimal trajectories converge to the point $x^*=(0.5619,0.2590,0.1790)$ which is a "projective fixed point" of the matrix $L(\alpha_2,\beta_2)$, i.e $\frac{x^*L(\alpha_2,\beta_2)}{x^*L(\alpha_2,\beta_2)e}= x^*$.

\section{Conclusion}
We showed that the competitive spectral radius of two families of matrices is continuous
under a positivity condition similar to the \emph{embedded cone} condition used by Jungers to show the continuity
of the lower spectral radius of a family of positive matrices. We provided
an algorithm to approximate the competive spectral radius, under the same assumption.
The lower spectral radius of family of invertible
matrices is continuous under a \emph{1-domination} or \emph{multicone}
condition (see~\cite{Bochi2014}, and also~\cite{BochiGourmelon,Avila2010}), it would be interesting to see whether this can be extended to the two-player case.
Understanding the continuity and approximability properties of the competitive spectral radius in more general situations is an interesting open question.

Another interesting problem would be to modify our numerical scheme (\Cref{algo-rvikm}), avoiding the evaluation of the ``full'' interpolation operator
$I_h^+$, which is the current bottleneck.


\printbibliography[heading=bibintoc]

\end{document}